\documentclass[11pt, reqno]{amsart}
\usepackage{amsmath, amsfonts, amsthm, amssymb, multicol, mathtools, dsfont}
\usepackage{graphicx, enumitem}
\usepackage{float, hyperref}
\usepackage{scalerel,stackengine}
\usepackage[usenames,dvipsnames]{xcolor}
\usepackage{pgfplots}
\pgfplotsset{width=10cm,compat=1.9}
\usepackage[square,sort,comma,numbers]{natbib}
\usepackage{fancyhdr}
\numberwithin{equation}{section}

\newtheorem{thm}{Theorem}[section]
\newtheorem{lma}[thm]{Lemma}
\newtheorem{cor}[thm]{Corollary}

\newtheorem{prop}[thm]{Proposition}

\renewcommand{\epsilon}{\varepsilon}

\newcommand{\rd}{\mathbb{R}^d}
\newcommand{\zd}{\mathbb{Z}^d}

\renewcommand{\geq}{\geqslant}
\renewcommand{\leq}{\leqslant}

\newcommand{\hd}{\dim_{\textup{H}}}
\newcommand{\frd}{\dim_{\textup{Fr}}}

\newcommand{\fs}{\dim^\theta_{\mathrm{F}}}
\newcommand{\fd}{\dim_{\mathrm{F}}}
\newcommand{\sd}{\dim_{\mathrm{S}}}

\newcommand{\R}{\mathbb{R}}

\newcommand{\Z}{\mathbb{Z}}

\newcommand{\N}{\mathbb{N}}

\renewcommand{\S}{\mathcal{S}_+}

\newcommand{\spt}{\mathrm{spt}\,}

\newcommand{\J}{\mathcal{J}}

\newcommand{\half}{\frac{1}{2}}

\makeatletter
\DeclareRobustCommand\widecheck[1]{{\mathpalette\@widecheck{#1}}}
\def\@widecheck#1#2{%
    \setbox\z@\hbox{\m@th$#1#2$}%
    \setbox\tw@\hbox{\m@th$#1%
       \widehat{%
          \vrule\@width\z@\@height\ht\z@
          \vrule\@height\z@\@width\wd\z@}$}%
    \dp\tw@-\ht\z@
    \@tempdima\ht\z@ \advance\@tempdima2\ht\tw@ \divide\@tempdima\thr@@
    \setbox\tw@\hbox{%
       \raise\@tempdima\hbox{\scalebox{1}[-1]{\lower\@tempdima\box
\tw@}}}%
    {\ooalign{\box\tw@ \cr \box\z@}}}
\makeatother

\stackMath
\newcommand\reallywidehat[1]{
\savestack{\tmpbox}{\stretchto{
  \scaleto{
    \scalerel*[\widthof{\ensuremath{#1}}]{\kern.1pt\mathchar"0362\kern.1pt}
    {\rule{0ex}{\textheight}}
  }{\textheight}
}{2.4ex}}
\stackon[-6.9pt]{#1}{\tmpbox}
}
\stackMath
\newcommand\reallywidecheck[1]{
\savestack{\tmpbox}{\stretchto{
  \scaleto{
    \scalerel*[\widthof{\ensuremath{#1}}]{\kern-.6pt\bigwedge\kern-.6pt}
    {\rule[-\textheight/2]{1ex}{\textheight}}
  }{\textheight}
}{0.5ex}}
\stackon[1pt]{#1}{\scalebox{-1}{\tmpbox}}
}

\definecolor{cite}{HTML}{2B77A4}
\definecolor{hyperlink}{HTML}{9E0D0D}
\hypersetup{
	colorlinks=true,
	linkcolor=hyperlink,
	urlcolor=hyperlink,
	citecolor=cite
}
\newcommand\numberthis{\addtocounter{equation}{1}\tag{\theequation}}

\makeatletter
\def\@setauthors{%
  \begingroup
  \def\thanks{\protect\thanks@warning}%
  \trivlist
  \centering\footnotesize \@topsep30\p@\relax
  \advance\@topsep by -\baselineskip
  \item\relax
  \author@andify\authors
  \def\\{\protect\linebreak}

  \normalsize\lowercase{\authors}%
  
	\ifx\@empty\contribs
  \else
    ,\penalty-3 \space \@setcontribs
    \@closetoccontribs
  \fi
  \endtrivlist
  \endgroup
}
\def\@settitle{\begin{center}
  \LARGE\lowercase{\@title}
    \end{center}%
  }
\makeatother
\newcommand{\authoremail}[1]{\email{\href{mailto:#1}{\color{cite}{#1}}}}
\newcommand{\authoraddress}[1]{\address{\normalfont{#1}}}

\title[Obtaining the Fourier spectrum via Fourier coefficients]{Obtaining the Fourier spectrum\\via Fourier coefficients}

\author{Marc Carnovale}
\authoraddress{Marc Carnovale}
\authoremail{carnom2@nationwide.com}

\author{Jonathan M. Fraser}
\authoraddress{J. M. Fraser, University of St Andrews, Scotland}
\authoremail{jmf32@st-andrews.ac.uk}
\thanks{JMF was financially supported by  a  \emph{Leverhulme Trust Research Project Grant} (RPG-2019-034).}

\author{Ana E. de Orellana}
\authoraddress{A. E. de Orellana, University of St Andrews, Scotland}
\authoremail{aedo1@st-andrews.ac.uk}
\thanks{AEdO was financially supported by the University of St Andrews.}
\date{}

\begin{document}
\maketitle
\thispagestyle{empty}

\begin{abstract}
  The Fourier spectrum is a family of dimensions that interpolates between the Fourier and Hausdorff dimensions and are defined in terms of certain energies which capture Fourier decay. In this paper we obtain a convenient discrete representation of those energies using the Fourier coefficients. As an example application, we use this representation to establish  sharp bounds for the Fourier spectrum of a general   measure with bounded support, improving previous estimates of the second-named author.
\end{abstract}
\tableofcontents

\section{Introduction}

Roughly speaking, the Hausdorff dimension of a measure quantifies its geometric scaling properties, while the Fourier dimension is more sensitive to  arithmetic structure,  smoothness, and  curvature.  For example, any measure supported on a hyperplane in $\rd$ has Fourier  dimension zero due to a lack of curvature relative to the ambient space,  and any measure on the middle third Cantor set has  Fourier dimension zero due to too much arithmetic resonance. On the other hand,  the surface measure on the  sphere in $\rd$ has Fourier dimension (and Hausdorff dimension) $d-1$ and random measures often  have large Fourier dimension (see for example \cite{salem}, \cite{bluhm} or \cite{PL11}). Dimension interpolation is the general approach of taking two related notions of fractal dimensions and studying them via a continuum of dimensions which live in-between.  One family of dimensions that interpolates between the Fourier and Hausdorff dimensions is the Fourier spectrum, introduced in \cite{Fra22}. The dimension interpolation approach is useful in this context to gain insight on the relation between the two dimensions and to obtain more information about the structure of measures than we would by studying each dimension in isolation.

The Fourier spectrum of a measure  is defined via certain energies which involve the Fourier transform of the measure.  These energies can be awkward to work with directly and the  purpose of this paper is to derive a simpler discrete expression for them which only requires understanding the Fourier coefficients, that is, the Fourier transform evaluated along the integers.  Our main result is Theorem~\ref{thm:coeffs}, which improves upon a partial result   derived in \cite[Proposition~3.1]{Fra22} which required  several unsatisfactory assumptions. In particular, we answer \cite[Question 3.2]{Fra22} in the affirmative. We demonstrate the utility of our main result by providing some simple applications.  For instance, we  obtain a general (and sharp) upper bound for the Fourier spectrum of a general measure, which states that the right semi-derivative of the Fourier spectrum at $0$ will never exceed the dimension of the ambient space (Proposition~\ref{propo:app1}).  Again, this improves on previous estimates from \cite{Fra22}, explicitly answering \cite[Question 1.4]{Fra22} in the negative.

\subsection{Notation} Throughout the paper we write $A\lesssim B$ whenever there exists a constant $c>0$ such that $A\leq cB$ and $A\approx B$  if $A\lesssim B$ and $B\lesssim A$. If the constant $c$ depends on some parameter $\lambda$ we wish to emphasise, we shall express it as a subscript on the corresponding symbol, e.g. $A\lesssim_{\lambda}B$. We adopt the convention that $A \lesssim \infty$ for all $A>0$ and  that $\infty \approx \infty$.

We write $\S(\rd)$ for the family of non-negative functions in the Schwartz class on $\rd$, i.e. 
functions $\psi:\rd\to\R_{\geq0}$ for which the partial derivatives of any order exist and decay at infinity faster than the reciprocal of any polynomial. That is, for all $N\in\N$ and all multi-indices $\alpha$, the inequality
\begin{equation*}
    \big| (\partial^{\alpha}\psi)(z) \big|\lesssim_{\alpha,N}\big( 1+|z| \big)^{-N},
\end{equation*}
holds for all $z\in\rd$. 

For $s>0$ we define the Riesz kernel as $\kappa_{s}(z) = |z|^{-s}$. Its Fourier and inverse Fourier transforms are, in the distributional sense and up to constants depending on $s$ and $d$, $\widehat{\kappa_{s}} = \kappa_{d-s}$ and $\widecheck{\kappa_{s}} = \kappa_{d-s}$, respectively (\cite[Theorem~2.4.6]{Gra14}).  We write $\spt \cdot$ to denote the support of a function or a measure. Given $x\in\rd$ and $r>0$, we write $B(x,r)$ for the open ball centred at $x$ and of radius $r$.

\section{Preliminaries}

\subsection{Energy integrals and the Fourier transform} Let us briefly introduce the relevant notions from Geometric Measure Theory. For a more thorough presentation the reader is referred to \cite{Fal03} and \cite{Mat15}.

It is well known that the \textit{Hausdorff dimension} of a Borel set $X\subseteq\rd$ can be defined as
\begin{equation*}
    \hd X = \sup\big\{s\leq d : I_{s}(\mu)<\infty \big\},
\end{equation*}
where the supremum is taken over all Borel measures $\mu$ supported on $X$ for which $0<\mu(X)<\infty$, and $I_{s}$ is the $s$-energy of the measure $\mu$ which can be expressed, for $s\in(0,d)$, as
\begin{equation}\label{eq:energy}
    I_{s}(\mu) = \iint \frac{d\mu (x)\,d\mu (y)}{|x-y|^{s}} \approx_{d,s} \int_{\rd}\big|\widehat{\mu}(z)\big|^2|z|^{s-d}\,dz.
\end{equation}

This connection  between the Hausdorff dimension of a set and the Fourier transform of measures it supports motivates the definition of the \emph{Fourier dimension} of a finite Borel measure $\mu$, which is
\begin{equation*}
    \fd\mu = \sup\big\{ s\geq0 : \big| \widehat{\mu}(z) \big|\lesssim |z|^{-\frac{s}{2}} \big\},
\end{equation*}
and the \emph{Fourier dimension} of a  Borel set  $X\subseteq\rd$, which is
\begin{equation*}
    \fd X = \sup\{ \fd\mu\leq d : \mu\text{ finite Borel measure on $X$}\}.
\end{equation*}

The definition of Hausdorff dimension for sets using energy integrals \eqref{eq:energy} gives rise to the \emph{Sobolev dimension} of a finite  Borel   measure, given by
\begin{equation*}
    \sd\mu = \sup\bigg\{ s\in\R : \int_{\rd}\big| \widehat{\mu}(z) \big|^2|z|^{s-d}\,dz<\infty \bigg\}.
\end{equation*}
These  notions of dimension are related by the inequalities $\fd\mu\leq\sd\mu $ for any finite  Borel measure $\mu$, and $\fd X\leq\hd X$ for any Borel set $X\subseteq\rd$.

\subsection{Fourier spectrum} The Fourier spectrum was introduced in \cite{Fra22} as a family of dimensions that, for measures, lie between the Fourier and Sobolev dimensions. To define them the author first introduced the $(s,\theta)$-energies for $\theta\in(0,1]$ and $s\geq0$ as
\begin{equation*}
    \J_{s,\theta}(\mu) = \bigg( \int_{\rd}\big| \widehat{\mu}(z) \big|^{\frac{2}{\theta}}|z|^{\frac{s}{\theta}-d}\,dz \bigg)^\theta,
\end{equation*}
and for $\theta =0$,
\begin{equation*}
    \J_{s,0}(\mu) = \sup_{z\in\rd}\big| \widehat{\mu}(z) \big|^2|z|^s.
\end{equation*}
Then the \emph{Fourier spectrum} of $\mu$ at $\theta$ is
\begin{equation*}
    \fs\mu = \sup\big\{ s\geq0 : \J_{s,\theta}(\mu)<\infty \big\}.
\end{equation*}

It is established in \cite[Theorem~1.1]{Fra22} that for any finite Borel measure $\mu$, $\fs\mu$ is non-decreasing and concave for all $\theta\in[0,1]$, and continuous for $\theta\in(0,1]$. Furthermore, if the measure $\mu$ is compactly supported, then $\fs\mu$ is also continuous at $\theta=0$. We refer the reader to \cite{Fra22} for a more comprehensive presentation of the Fourier spectrum.

\subsection{Fourier coefficients and energy integrals}

In \cite{HR03} an alternative representation for the $s$-energy of a measure was derived, this time in terms of its Fourier coefficients. If $\mu$ is a finite Borel measure on $\rd$ with support contained in $[0,1]^d$, then for all $s\in(0,d)$,
\begin{equation}\label{eq:senergycoeffs}
    I_{s}(\mu) \approx_{d,s}\big| \widehat{\mu}(0) \big|^2 + \sum_{z\in\zd\backslash\{ 0 \}}\big| \widehat{\mu}(z) \big|^2|z|^{s-d}.
\end{equation}

A similar formula for the $(s,\theta)$-energy exists (\cite[Proposition 3.1]{Fra22}), but was only established in the case where $\theta$ is  the reciprocal of a positive integer and $0<s<d\theta$. In this  case,
\begin{equation*}
    \J_{s,\theta}(\mu) \approx \bigg( \big| \widehat{\mu}(0) \big|^{\frac{2}{\theta}} + \sum_{z\in\theta\Z^d\backslash\{ 0 \}}\big| \widehat{\mu}(z) \big|^\frac{2}{\theta}|z|^{\frac{s}{\theta}-d} \bigg)^\theta.
\end{equation*}
This result has a number of drawbacks.  It is unsatisfactory to sum over $\theta\Z^d$ and not $\Z^d$ and to require $\theta$ to be the reciprocal of an integer, but it is perhaps even more unsatisfactory to require $s<d\theta$ because for measures with positive Fourier dimension, it is necessary to study the energies for $s>d\theta$ in order to get a complete picture.   The main objective of this paper is to prove Theorem~\ref{thm:coeffs}, stated below, where we address all the aforementioned drawbacks  in the case where the measure is supported on a closed subset of  $(0,1)^d$.

\section{Main result: obtaining the Fourier spectrum via Fourier coefficients}

The objective of this section is to state and prove the main theorem of the paper.
\begin{thm}\label{thm:coeffs}
  Let $\mu$ be a finite Borel measure on $\rd$ with support contained in $[\delta,1-\delta]^d$ for some $0<\delta<1/2$. Then for all $\theta\in(0,1]$ and $s>0$,
  \begin{equation}\label{eq:FourierCoeffs}
      \J_{s,\theta}(\mu)\approx_{\delta}\bigg( \big| \widehat{\mu}(0) \big|^{\frac{2}{\theta}} + \sum_{z\in\zd\backslash\{ 0 \}}\big| \widehat{\mu}(z) \big|^{\frac{2}{\theta}}|z|^{\frac{s}{\theta}-d} \bigg)^\theta.
  \end{equation}
In particular, 
\[
\fs \mu = \sup\bigg\{ s\in\R :\sum_{z\in\zd\backslash\{ 0 \}}\big| \widehat{\mu}(z) \big|^{\frac{2}{\theta}}|z|^{\frac{s}{\theta}-d} <\infty \bigg\}.
\]
\end{thm}

The requirement that there should exist such a $\delta$ in the above is necessary.  To see this, let $d=1$ and $\mu$ be the Lebesgue measure on $[0,1]$.  Then $ \widehat{\mu}(z) = 0$ for all $z \in \mathbb{Z} \backslash\{ 0 \}$ but the Sobolev dimension (and thus Fourier spectrum) is finite; in fact $\fs \mu = 2$ for all $\theta \in [0,1]$.  The point here is that sampling the Fourier transform along the integers does not ``see'' the relevant fluctuations. 

We note that the analogous result  indeed holds in the $\theta=0$ case, dating back to Kahane \cite{kahane}.  In particular, if $\mu$ is a measure    supported in $[\delta,1-\delta]^d$ for some $0<\delta<1/2$ and $t>0$ is such that
\[
\big| \widehat{\mu}(z) \big| = O(|z|^{-t}) 
\]
for $z \in \mathbb{Z}^d$, then the same holds (possibly with a different implicit constant) for $z \in \mathbb{R}^d$, see \cite[Lemma~1, page 252]{kahane}. Therefore, the Fourier dimension of a measure can be observed by sampling only along the integers.  Again the condition that the measure is  supported in $[\delta,1-\delta]^d$ is necessary and again this can be seen by the Lebesgue measure example.  

\subsection{The Fourier transform and the Laplacian operator}

To prove Theorem~\ref{thm:coeffs} we will replace $|z|^{\frac{s}{\theta}-d}$ by a suitable kernel involving the Laplacian operator. Recall that the Laplacian operator is defined for sufficiently smooth functions $\psi$ on $\R^d$ as
\begin{equation*}
    \psi\mapsto\Delta \psi = \sum_{j=1}^d \frac{\partial^2\psi}{\partial x_{j}^2}.
\end{equation*}
Fourier transforms behave particularly well with respect to differential operators. In what follows we shall only consider non-negative functions $\psi$ in the Schwartz class, but the same results hold in more general contexts.

From the definition of the Fourier transform and the integration by parts formula it is easy to see that
\begin{equation*}
    \widehat{\Delta \psi}(z) = (-4\pi|z|^2)\,\widehat{\psi}(z).
\end{equation*}
Iterating, for  higher powers of the Laplacian we get
\begin{equation}\label{eq:LaplacianInteger}
     \reallywidehat{(-\Delta)^k\psi}(z)  \approx |z|^{2k}\, \widehat{\psi}(z),
\end{equation}
where we replace $\Delta$ with $-\Delta$ in order to make the constants positive and maintain consistency with the $\approx$ notation. This leads to a generalisation  to  (non-integer) powers of the Laplacian. For $k\geq0$, the \emph{fractional Laplacian operator} $(-\Delta)^k$ is defined   by extension as
\begin{equation}\label{eq:Laplacian}
    \reallywidehat{(-\Delta)^k\psi}(z) = |z|^{2k}\, \widehat{\psi}(z).
\end{equation}
Note that when $k\in\N$, \eqref{eq:Laplacian} is not the same as, but is  comparable to,  \eqref{eq:LaplacianInteger}. However, since the constants involved are unimportant for our objective (they do depend on $k$ but not on $z$ or $\psi$), we adopt \eqref{eq:Laplacian} as the definition for all $k \geq 0$. We refer the reader to the survey article \cite{Aba21} for a more detailed discussion about the fractional Laplacian.

\subsection{Some preliminary estimates involving $L^p$ norms}

The proof of Theorem~\ref{thm:coeffs} will use three technical lemmas which we  prove in this subsection. The first lemma shows that for the Fourier transform of certain measures, the $\ell^p(\Z^d)$ and $L^p(\rd)$ norms are equivalent. This lemma can be found in \cite[Lemma~4.4]{Car16}, but we prove it here for completeness and to keep notation consistent.

  \begin{lma}\label{lemma1}
    Let $\nu$ be a measure with $\spt\nu\subseteq[0,1]^d$. Then for all $p\in[1,\infty]$,
    \begin{equation*}
        \Bigg( \sum_{z\in \zd}\big| \widehat{\nu}(z) \big|^p \Bigg)^{1/p} \lesssim \bigg( \int_{\rd} \big| \widehat{\nu}(z) \big|^p\,dz\bigg)^{1/p},
    \end{equation*}
    and if $\spt\nu\subseteq[\delta,1-\delta]^d$ for some $0<\delta<1/2$, then,
    \begin{equation*}
        \bigg( \int_{\rd}\big| \widehat{\nu}(z) \big|^p\,dz \bigg)^{1/p}\approx_{\delta}\Bigg( \sum_{z\in\zd}\big| \widehat{\nu}(z) \big|^p \Bigg)^{1/p}.
    \end{equation*}
  \end{lma}
\begin{proof}
  For $p=\infty$, the first inequality is immediate. Thus, we consider $p<\infty$.

  Let $g\in\S(\rd)$ be such that $g(x) =1$ for $x\in[0,1]^d$ and for which $\spt \widehat{g}\subseteq \big[ -\half,\half \big]^d$. Then, since $\nu = g\nu$,
  \begin{equation*}
      \big| \widehat{\nu}(z) \big| = \big| \widehat{\nu}*\widehat{g}(z) \big| = \Bigg| \int_{\big[ -\half,\half \big]^d}\widehat{\nu}(z-\eta)\widehat{g}(\eta)\,d\eta \Bigg| \leq \| \widehat{g}\|_{L^\infty(\rd)}\int_{\big[ -\half,\half \big]^d}\big| \widehat{\nu}(z-\eta) \big|\,d\eta.
  \end{equation*}
  Thus, by H\"older's inequality,
  \begin{align*}
      \sum_{z\in\zd}\big| \widehat{\nu}(z) \big|^p&\lesssim \sum_{z\in\zd}\Bigg( \int_{\big[ -\half,\half \big]^d}\big| \widehat{\nu}(z-\eta)\,d\eta \big| \Bigg)^p\\
&\leq\sum_{z\in\zd}\int_{\big[ -\half,\half \big]^d}\big| \widehat{\nu}(z-\eta) \big|^p\,d\eta\\
& = \int_{\rd}\big| \widehat{\nu}(z) \big|^p\,dz.
  \end{align*}
  Now we prove the second part. Let $h\in\S(\rd)$ be such that $h(x) = 1$ for $x\in[\delta,1-\delta]^d$ and $\spt h\subseteq[0,1]^d$. Then, for all $z\in\rd$, \cite[(111)]{Wol03} yields
  \begin{equation} \label{wolffa}
      \widehat{\nu}(z) = \widehat{h\nu}(z) = \sum_{n\in\zd}\widehat{\nu}(n)\widehat{h}(z-n).
  \end{equation}
  Thus, Fubini's theorem and the translation invariance of the $L^1(\rd)$ norm give
  \begin{align*} 
    \|\widehat{\nu}\|_{L^1(\rd)} &= \int_{\rd}\Big| \sum_{n\in\zd}\widehat{\nu}(n)\widehat{h}(z-n) \Big|\,dz \\
&\leq \sum_{n\in\zd}\big| \widehat{\nu}(n) \big|\int_{\rd}\big| \widehat{h}(z-n) \big|\,dz \\
&\leq\|\widehat{\nu}\|_{\ell^1(\zd)}\|\widehat{h}\|_{L^1(\rd)}.\numberthis\label{eq:boundnuh1}
  \end{align*}
  Using the first part of the lemma, for all $z\in\rd$,
  \begin{equation*}
      \sum_{n\in\Z^d}\big| \widehat{h}(z-n) \big|\lesssim \int_{\rd}  | \widehat{h}(z-n) | \,dn = \|\widehat{h}\|_{L^1(\rd)}.
  \end{equation*}
  Then, by \eqref{wolffa} and  H\"older's inequality,
  \begin{align*}
    \|\widehat{\nu}\|_{L^\infty(\rd)} &= \sup_{z\in\rd}\Big| \sum_{n\in\zd}\widehat{\nu}(n)\widehat{h}(z-n) \Big|\\
    &\leq \sup_{z\in\rd}\|\widehat{\nu}\|_{\ell^{\infty}(\zd)}\sum_{n\in\zd}\big| \widehat{h}(z-n) \big|\\
    &\lesssim \|\widehat{\nu}\|_{\ell^\infty(\zd)}\,\|\widehat{h}\|_{L^1(\rd)}.\numberthis\label{eq:boundnuhinf}
  \end{align*}
  Finally, using the Riesz--Thorin interpolation theorem (\cite[Theorem~1.3.4]{Gra14}) with the operator being the convolution with $\widehat h$, we can interpolate between \eqref{eq:boundnuh1} and \eqref{eq:boundnuhinf}. The result is that  for all  $p\in(1,\infty)$,
  \begin{equation*}
    \|\widehat{\nu}\|_{L^p(\rd)}=\|\widehat{\nu}*\widehat{h}\|_{L^p(\rd)}\leq \|\widehat{\nu}\|_{\ell^{p}(\zd)}\|\widehat{h}\|_{L^1(\rd)} \approx  \|\widehat{\nu}\|_{\ell^{p}(\zd)},
  \end{equation*}
  completing the proof.
\end{proof}

The following two lemmas will be used to guarantee that the $L^p$ and $\ell^p$ errors obtained when replacing the Riesz kernel with a compactly supported approximation of it are bounded.

\begin{lma}\label{lemma2}
  Let $\mu$ be a finite Borel measure supported in $[0,1]^d$, and $\psi\in\S(\rd)$ be compactly supported. For $\varepsilon>0$ let $\rho_{\varepsilon}\in\S(\rd)$ with $\spt\rho_{\varepsilon}\subseteq B(0,\varepsilon)$ and such that $\rho_{\varepsilon}(z)=1$ for $z\in B(0,\varepsilon/2)$. Then, for any $t\in(0,d)$, $k\geq0$, and $\theta\in(0,1]$ such that $\theta>\frac{2t-4k}{d}$,
  \begin{equation}\label{eq:lemma2}
    \Big\| \reallywidehat{\big( (-\Delta)^k(\psi) \big)*\mu * \widecheck{\kappa_{t}}}  - \reallywidehat{\big( (-\Delta)^k(\psi) \big)*\mu * (\rho_{\varepsilon}\widecheck{\kappa_{t}})}\Big\|_{L^{2/\theta}(\rd)}\lesssim_\varepsilon 1.
  \end{equation}
\end{lma}
\begin{proof}
  Let $p \coloneqq \frac{2}{\theta}$ and  $p' \coloneqq \frac{2}{2-\theta}$ be conjugate exponents. First note that from \eqref{eq:Laplacian} and the linearity of both the Fourier transform and the convolution, we obtain
\begin{align*}
  \reallywidehat{\big( (-\Delta)^k(\psi) \big)*\mu * \widecheck{\kappa_{t}}}  - \reallywidehat{\big( (-\Delta)^k(\psi) \big)*\mu * (\rho_{\varepsilon}\widecheck{\kappa_{t}})}&=\reallywidehat{\big((-\Delta)^k(\psi)\big) * \mu * \big( \widecheck{\kappa_{t}}(1-\rho_{\varepsilon}) \big)} \\
  &= \reallywidehat{(-\Delta)^k(\psi)}\,\widehat{\mu}\,\reallywidehat{\widecheck{\kappa_{t}}(1-\rho_{\varepsilon})}\\
  &= |z|^{2k}\,\widehat{\psi}\,\widehat{\mu}\,\reallywidehat{\widecheck{\kappa_{t}}(1-\rho_{\varepsilon})}\\
  &= \reallywidehat{\psi*\mu}\,|z|^{2k}\,\reallywidehat{\widecheck{\kappa_{t}}(1-\rho_{\varepsilon})}\\
  &= \reallywidehat{\psi*\mu}\,\reallywidehat{(-\Delta)^k \big(\widecheck{\kappa_{t}}(1-\rho_{\varepsilon  })\big)}\\
  &= \reallywidehat{(\psi*\mu)*\big( (-\Delta)^k \big(\widecheck{\kappa_{t}}(1-\rho_{\varepsilon  })\big) \big)}.\numberthis\label{eq:lemma2equation}
\end{align*}

Since $p'\in[1,2]$,  first applying the Hausdorff--Young inequality for the conjugate exponents $p$ and $p'$, and then Young's inequality for convolutions, we have, using \eqref{eq:lemma2equation} in the left-hand side of \eqref{eq:lemma2},
\begin{align*}
  \Big\| \reallywidehat{\big( (-\Delta)^k(\psi) \big)*\mu * \widecheck{\kappa_{t}}}  - \reallywidehat{\big( (-\Delta)^k(\psi) \big)*\mu * (\rho_{\varepsilon}\widecheck{\kappa_{t}})}&\Big\|_{L^{2/\theta}(\rd)} \\
  &= \Big\| \reallywidehat{(\psi*\mu) * \big( (-\Delta)^k\big( \widecheck{\kappa_{t}}(1-\rho_{\varepsilon}) \big) \big)}\Big\|_{L^p(\rd)} \\
  &\leq \big\| (\psi*\mu) * \big((-\Delta)^k\big( \widecheck{\kappa_{t}}(1-\rho_{\varepsilon}) \big)\big)\big\|_{L^{p'}(\rd)}\\
  &\leq \big\| \psi*\mu \big\|_{L^{1}(\rd)} \big\| (-\Delta)^k\big( \widecheck{\kappa_{t}}(1-\rho_{\varepsilon}) \big)\big\|_{L^{p'}(\rd)}.
\end{align*}
For the first term, Young's inequality for convolutions gives
\begin{equation*}
  \big\| \psi*\mu \big\|_{L^1(\rd)}\leq  \big\| \psi \big\|_{L^1(\rd)}\mu(\rd) \lesssim 1
\end{equation*}
For the second, note that $\widecheck{\kappa_{t}}(1-\rho_{\varepsilon})$ is $\widecheck{\kappa_{t}}$ on $\rd\backslash B(0,\varepsilon)$, and $0$ on $B(0,\varepsilon/2)$. For any $s>0$, $(-\Delta)^k\kappa_{s} = \kappa_{s+2k}$, in the distributional sense and up to constants depending on $s$ and $d$. This, and the triangle inequality, yields
\begin{equation*}
\big\| (-\Delta)^k\big( \widecheck{\kappa_{t}}(1-\rho_{\varepsilon}) \big)\big\|_{L^{p'}(\rd)}\leq \big\|\kappa_{d-t+2k}\big\|_{L^{p'}(\rd\backslash B(0,\varepsilon))} + \big\|  (-\Delta)^k\big( \widecheck{\kappa_{t}}(1-\rho_{\varepsilon}) \big)\big\|_{L^{p'}(B(0,\varepsilon)\backslash B(0,\varepsilon/2))}.
\end{equation*}
Since $p'>\frac{d}{d-t+2k}$, 
\begin{equation*}
\big\|\kappa_{d-t+2k}\big\|_{L^{p'}(\rd\backslash B(0,\varepsilon))} \lesssim \varepsilon^{-(d-t+2k)+1/p'} \lesssim_\varepsilon 1.
\end{equation*}
Finally, recalling that  $k$ is fixed, $(-\Delta)^k\big( \widecheck{\kappa_{t}}(1-\rho_{\varepsilon}) \big)\in L^\infty\big(B(0,\varepsilon)\backslash B(0,\varepsilon/2)\big)$. Given that the restriction of the $d$-dimensional Lebesgue measure to $B(0,\varepsilon)\backslash B(0,\varepsilon/2)$ is finite, then for any $q<\infty$, $(-\Delta)^k\big( \widecheck{\kappa_{t}}(1-\rho_{\varepsilon}) \big)\in L^q\big(B(0,\varepsilon)\backslash B(0,\varepsilon/2)\big)$, in particular this holds for $q=p'$, which gives the result.
\end{proof}
\begin{lma}\label{lemma:lemma2sum}
  Let $\mu$ be a finite Borel measure supported in $[0,1]^d$ and $\psi\in\S(\rd)$ be compactly supported. For $\varepsilon>0$ let  $\rho_{\varepsilon}\in\S(\rd)$ with $\spt\rho_{\varepsilon}\subseteq B(0,\varepsilon)$ and such that $\rho_{\varepsilon}(z)=1$ for $z\in B(0,\varepsilon/2)$. Then, for any $t\in(0,d)$, $k\geq0$, and $\theta\in(0,1]$ such that $\theta>\frac{2t-4k}{d}$,
  \begin{equation*}
    \Big\| \reallywidehat{\big( (-\Delta)^k(\psi) \big)*\mu * \widecheck{\kappa_{t}}}  - \reallywidehat{\big( (-\Delta)^k(\psi) \big)*\mu * (\rho_{\varepsilon}\widecheck{\kappa_{t}})}\Big\|_{\ell^{2/\theta}(\zd\backslash\{ 0 \})} \lesssim_\varepsilon 1.
  \end{equation*}
\end{lma}
\begin{proof}
  The proof follows that of Lemma~\ref{lemma2} after noting that the Hausdorff--Young and Young inequalities can also be applied to the counting measure on $\zd\backslash\{ 0 \}$, where we remove $0$ from the domain so that $\kappa_{t}$ is well-defined.
\end{proof}

\subsection{Proof of Theorem~\ref{thm:coeffs}}
In this subsection we give the  proof of the main theorem.   We will first assume that the left-hand side of \eqref{eq:FourierCoeffs} is finite, and prove that the right-hand side is also finite and both are comparable.

  Let $\psi\in\S(\rd)$ be a function supported in $(0,1)^d$ with $\|\psi\|_{L^1(\rd)} = 1$, and define an  approximate identity $(\psi_{n})_{n\in\N}$ by $\psi_{n}(z)\coloneqq n^{d}\psi(n x)$. Then $\widehat{\psi}_{n}(z) \to \widehat{\psi}(0) = 1$ and $\widehat{\psi}_{n}\widehat{\mu}\to\widehat{\mu}$ uniformly as $n\to\infty$. Thus, by the convolution formula and dominated convergence theorem, 
  \begin{equation}\label{eq:limitn}
      \J_{s,\theta}(\psi_{n}*\mu)^{1/\theta} = \int_{\rd}\big| \widehat{\psi}_{n}\widehat{\mu}(z) \big|^{\frac{2}{\theta}}|z|^{\frac{s}{\theta}-d}\,dz \overset{n\to\infty}{\longrightarrow}\int_{\rd}\big| \widehat{\mu}(z) \big|^{\frac{2}{\theta}}|z|^{\frac{s}{\theta}-d}\,dz = \J_{s,\theta}(\mu)^{1/\theta}<\infty.
  \end{equation}
  Therefore, it suffices to work with $\J_{s,\theta}(\psi_{n}*\mu)$ and  $n$ sufficiently large.

  Choose $k\geq0$ such that $\frac{s-d\theta}{2}\in(2k-d,2k)$ and define $s'$ to satisfy
  \begin{equation*}
      \frac{s-d\theta}{2} = 2k+\frac{s'-d\theta}{2}.
  \end{equation*}
  Then $t\coloneqq \frac{d\theta-s'}{2}\in(0,d)$. For each $n\in\N$ define the functions $\nu_{n}$ by
  \begin{equation*}
      \nu_{n} = \big( (-\Delta)^k(\psi_{n}) \big)*\mu.
  \end{equation*}
  Using \eqref{eq:Laplacian} we have
  \begin{equation*}
    \big|\widehat{\nu}_{n}(z) \big| = |z|^{2k}\, \big|\widehat{\psi}_{n}(z)\widehat{\mu}(z) \big|.
  \end{equation*}
  Therefore, since $\frac{s'}{\theta}-d = \frac{s}{\theta}-d-\frac{4k}{\theta}$,
  \begin{align*}
    \J_{s,\theta}(\psi_{n}*\mu)^{1/\theta} &=\int_{\rd}\big|\reallywidehat{\psi_{n}*\mu }(z) \big|^{\frac{2}{\theta}} |z|^{\frac{s}{\theta}-d}\,dz\\
    &=\int_{\rd}\big| \widehat{\nu}_{n}(z)\big|^{\frac{2}{\theta}}|z|^{-\frac{4k}{\theta}}|z|^{\frac{s}{\theta}-d}\,dz\\
    &=\int_{\rd}\big| \widehat{\nu}_{n}(z) \big|^{\frac{2}{\theta}}|z|^{\frac{s'}{\theta}-d}\,dz\\
    &= \J_{s',\theta}(\nu_{n})^{1/\theta}\\
    &=\int_{\rd}\big| \reallywidehat{\nu_{n}*\widecheck{\kappa_{t}}}(z) \big|^\frac{2}{\theta}\,dz.
  \end{align*}

  Given $\varepsilon>0$, define $\rho_{\varepsilon}\in\S(\rd)$ with $\spt\rho_{\varepsilon}\subseteq B(0,\varepsilon)$ and such that $\rho_{\varepsilon}(z) = 1$ for $z\in B(0,\varepsilon/2)$. Since $s>0$, then $\frac{\theta}{2}>\frac{t-2k}{d}$ and we may use Lemma~\ref{lemma2} with   $\psi =\psi_{n}$, to get
  \begin{equation*}
    \int_{\rd}\big| \reallywidehat{\nu_{n}*\widecheck{\kappa_{t}}}(z) \big|^\frac{2}{\theta}\,dz \approx \int_{\rd} \big| \reallywidehat{\nu_{n}*(\rho_{\varepsilon}\widecheck{\kappa_{t}})}(z) \big|^{\frac{2}{\theta}}\,dz.
  \end{equation*}
   Note that $\kappa_{t}(z)$ is only well-defined for $z\neq0$, but this does not affect the above integrals in a meaningful way. Given  $n\in\N$, choose  $\varepsilon$ sufficiently small to ensure that the function $\nu_{n}* (\rho_{\varepsilon}\widecheck{\kappa_{t}})$ is supported on a closed subset of $(0,1)^d$. Therefore, using the second part of Lemma~\ref{lemma1} with the density $\nu = \nu_{n}*(\rho_{\varepsilon}\widecheck{\kappa_{t}})$,
   \begin{equation*}
       \Big\| \reallywidehat{\nu_{n}*(\rho_{\varepsilon}\widecheck{\kappa_{t}})} \Big\|_{L^{2/\theta}(\rd)}^{\frac{2}{\theta}}\approx \Big\| \reallywidehat{\nu_{n}*(\rho_{\varepsilon}\widecheck{\kappa_{t}})} \Big\|_{\ell^{2/\theta}(\zd)}^{\frac{2}{\theta}}
   \end{equation*}
and, further,   Lemma~\ref{lemma:lemma2sum} gives
  \begin{equation*}
    \Big\| \reallywidehat{\nu_{n}*(\rho_{\varepsilon}\widecheck{\kappa_{t}})} \Big\|_{\ell^{2/\theta}(\zd)}^{\frac{2}{\theta}} \approx \Big\|  \reallywidehat{\nu_{n}*\widecheck{\kappa_{t}}}\Big\|^{\frac{2}{\theta}}_{\ell^{2/\theta}(\zd)}.
  \end{equation*}
Note that
\begin{align*}
  \sum_{z\in\zd}\big| \reallywidehat{\nu_{n}*\widecheck{\kappa_{t}}}(z) \big|^{\frac{2}{\theta}} &=\big| \reallywidehat{\nu_{n}*\widecheck{\kappa_{t}}}(0) \big|^{\frac{2}{\theta}} +  \sum_{z\in\zd\backslash\{ 0 \}}\big| \widehat{\nu}_{n}(z) \big|^{\frac{2}{\theta}} |z|^{\frac{s'}{\theta}-d}\\
& = \big| \reallywidehat{(-\Delta)^k(\psi_{n})*\mu*\widecheck{\kappa_{t}}}(0) \big|^{\frac{2}{\theta}}  + \sum_{z\in\zd\backslash\{ 0 \}}\big| \reallywidehat{\psi_{n}*\mu}(z) \big|^{\frac{2}{\theta}}|z|^{\frac{s}{\theta}-d}.
\end{align*}
Using that $\psi_{n}\in\S(\rd)$, an easy calculation gives
\[
\big| \reallywidehat{(-\Delta)^k(\psi_{n})*\mu*\widecheck{\kappa_{t}}}(0) \big|  =\big| \reallywidehat{(-\Delta)^k(\psi_{n})* \widecheck{\kappa_{t}}}(0) \big| \,  | \widehat {\mu} (0) | \approx | \widehat {\mu} (0) |
\]
and thus,
\begin{equation*}
  \Big\| \reallywidehat{\nu_{n}*(\rho_{\varepsilon}\widecheck{\kappa_{t}})} \Big\|_{\ell^{2/\theta}(\zd)}^{\frac{2}{\theta}} \approx \big|\widehat{\mu} (0)\big|^{\frac{2}{\theta}} +    \sum_{z\in\zd\backslash\{ 0 \}}\big| \reallywidehat{\psi_{n}*\mu}(z) \big|^{\frac{2}{\theta}}|z|^{\frac{s}{\theta}-d}   \overset{n\to\infty}{\longrightarrow}  \big|\widehat{\mu} (0)\big|^{\frac{2}{\theta}} +    \sum_{z\in\zd}\big| \widehat{\mu}(z) \big|^{\frac{2}{\theta}}|z|^{\frac{s}{\theta}-d}
\end{equation*}
again by the convolution formula and the dominated convergence theorem since $\psi_n \to 1$ uniformly as $n \to \infty$.  Together with \eqref{eq:limitn}, this gives the desired result in the case where $\J_{s,\theta}(\mu)<\infty$.

For the converse note that if instead we assume
\begin{equation*}
  \bigg( \big| \widehat{\mu}(0) \big|^{\frac{2}{\theta}} + \sum_{z\in\zd\backslash\{ 0 \}}\big| \widehat{\mu}(z) \big|^{\frac{2}{\theta}}|z|^{\frac{s}{\theta}-d} \bigg)^\theta <\infty,
\end{equation*}
then the proof follows similarly using first Lemma~\ref{lemma:lemma2sum} to get
\begin{equation*}
 \sum_{z\in\zd}\big| \reallywidehat{\psi_{n}*\mu}(z) \big|^{\frac{2}{\theta}}|z|^{\frac{s}{\theta}-d}\approx \big\| \reallywidehat{\nu_{n}*(\rho_{\varepsilon}\widecheck{\kappa_{t}})}  \big\|_{\ell^{2/\theta}(\zd)}^{\frac{2}{\theta}}.
\end{equation*}
Then by Lemma~\ref{lemma1},
\begin{equation*}
  \Big\| \reallywidehat{\nu_{n}*(\rho_{\varepsilon}\widecheck{\kappa_{t}})}  \Big\|_{\ell^{2/\theta}(\zd)}^{\frac{2}{\theta}} \approx \Big\| \reallywidehat{\nu_{n}*(\rho_{\varepsilon}\widecheck{\kappa_{t}})}  \Big\|_{L^{2/\theta}(\rd)}^{\frac{2}{\theta}}.
\end{equation*}
Finally, Lemma~\ref{lemma2} proves the result.

\section{Applications}

\subsection{A discrete formulation for general bounded measures}

One (unavoidable) drawback of the main theorem is that the measure must be supported on a region which is bounded away from the boundary of the unit cube.  Theoretically this is not a problem since measures with bounded support can always be scaled to satisfy this requirement.  However, for certain examples which come  supported on, e.g.~the whole unit cube, it can be inconvenient. Nonetheless, an explicit discrete formulation is still possible for these measures, but one is forced to consider different Fourier coefficients, in particular, evaluated along $\alpha \zd$ instead of $\zd$ for some $\alpha>0$.

\begin{cor}\label{cor:transformation}
  Let $\mu$ be a finite Borel measure on $\rd$ with compact support of diameter $R>0$. Then for all $0<\alpha<1/R$ it holds that  for all $\theta\in(0,1]$,
  \begin{equation*}
      \J_{s,\theta}(\mu) \approx_{\alpha} \bigg(\big| \widehat{\mu}(0) \big|^{\frac{2}{\theta}} + \sum_{z\in \alpha\zd\backslash\{ 0 \}}\big| \widehat{\mu}(z) \big|^{\frac{2}{\theta}}|z|^{\frac{s}{\theta}-d}\bigg)^\theta.
  \end{equation*}
In particular, 
\[
\fs \mu = \sup\bigg\{ s\in\R :\sum_{z\in\alpha\zd\backslash\{ 0 \}}\big| \widehat{\mu}(z) \big|^{\frac{2}{\theta}}|z|^{\frac{s}{\theta}-d} <\infty \bigg\}
\]
and also (for completeness)
\[
\fd \mu =  \sup\bigg\{ s\in\R : \sup_{z\in\alpha\zd }\big| \widehat{\mu}(z) \big|^{2}|z|^{s} <\infty \bigg\}.
\]
\end{cor}
\begin{proof}
Let $0<\alpha<1/R$.  Then there exists a translation  $\beta\in\R^d$ such that the function $\phi$ defined as $\phi(z) = \alpha z + \beta$ satisfies that $\phi\big( \spt\mu \big) \subseteq [\delta,1-\delta]^d$ for some $0<\delta<1/2$ depending only on $\alpha$. Then, writing $\phi(\mu) = \mu \circ \phi^{-1}$, $\big| \widehat{\phi(\mu) }(z)\big| \approx_\alpha \big|\widehat{\mu}(\alpha  z) \big|$.  Therefore, by a simple change of variables,
\[
 \J_{s,\theta}\big(\mu \big) \approx_\alpha  \J_{s,\theta}\big(\phi(\mu) \big) 
\]
 and,  by Theorem~\ref{thm:coeffs},
  \begin{align*}
      \J_{s,\theta}\big(\phi(\mu) \big)^{1/\theta} &\approx_\alpha \big| \widehat{\phi(\mu) }(0) \big|^{\frac{2}{\theta}} + \sum_{z\in\zd\backslash\{ 0 \}}\big| \widehat{\phi(\mu) }(z) \big|^{\frac{2}{\theta}}|z|^{\frac{s}{\theta}-d}\\
      &\approx_{\alpha} \big| \widehat{\mu}(0) \big|^{\frac{2}{\theta}} + \sum_{z\in\zd\backslash\{ 0 \}}\big| \widehat{\mu}(\alpha z) \big|^{\frac{2}{\theta}} |z|^{\frac{s}{\theta}-d}\\
      &\approx_{\alpha}  \big| \widehat{\mu}(0) \big|^{\frac{2}{\theta}}  + \sum_{z\in \alpha\zd\backslash\{ 0 \}}\big| \widehat{\mu}(z) \big|^{\frac{2}{\theta}}|z|^{\frac{s}{\theta}-d},
  \end{align*}
which proves the first claim.  The formulation of the Fourier spectrum for $\theta \in (0,1]$ is then immediate. For the case $\theta=0$, the formulation of the Fourier dimension follows from \cite[Lemma~1, page~252]{kahane}.  In particular, if    $t>0$ is such that
\[
\big| \widehat{\phi(\mu)}(z) \big|    \approx_\alpha \big|\widehat{\mu}(\alpha z) \big|   \lesssim |\alpha z|^{-\frac{t}{2}}\approx_\alpha |  z |^{-\frac{t}{2}}
\]
for   $z \in \mathbb{Z}^d$, then   $\big|\widehat{\mu}(z) \big| \lesssim  |  z |^{-\frac{t}{2}}$ for $z \in \mathbb{R}^d$. 
\end{proof}

\subsection{General bounds for  the Fourier spectrum}

If $\mu$ is a finite measure with compact support then it is easy to see that $\widehat{\mu}$ is Lipschitz. Using this fact, it was shown in \cite[Theorem~1.3]{Fra22} that
\begin{equation*}
    \fs\mu\leq\fd\mu+d\Big( 1+\frac{\fd\mu}{2} \Big)\theta.
\end{equation*}
It was left as an open question whether this bound is sharp in general, \cite[Question~1.4]{Fra22}.  We prove that it is not sharp, by proving a significantly better general bound in the case when the Fourier dimension is strictly positive. 

\begin{prop}\label{propo:app1}
  Let $\mu$ be a finite Borel measure with compact support on $\rd$. Then for all $\theta\in[0,1]$,
  \begin{equation}\label{eq:propo:app1}
      \fs\mu \leq \fd\mu + d\theta.
  \end{equation}
\end{prop}
\begin{proof}
  For $\theta=0$ or for  $\fd \mu = \infty$ the result is trivial, so we only consider the other cases.  Let $R<\infty$ denote the diameter of the support of $\mu$.   Let $\theta\in(0,1]$ and  $\fd\mu<t<\infty$.  It then follows from Corollary~\ref{cor:transformation} that, for some $0<\alpha<1/R$, there exists a sequence 
 $(z_{n})_{n\in\N}\subseteq \alpha\zd$ with $|z_{n}|\to\infty$ and such that for all $n\in\N$,
  \begin{equation*}
      \big| \widehat{\mu}(z_{n}) \big|\gtrsim|z_{n}|^{-\frac{t}{2}}.
  \end{equation*}
  Then, using Corollary~\ref{cor:transformation} again,
  \begin{align*}
    \J_{s,\theta}(\mu)^{1/\theta} &\approx \big| \widehat{\mu}(0) \big|^{\frac{2}{\theta}} + \sum_{ z\in\alpha\zd\backslash\{ 0 \}}\big| \widehat{\mu}(z) \big|^{\frac{2}{\theta}}|z|^{\frac{s}{\theta}-d}\\
    &\gtrsim \sum_{n\in\N}|z_{n}|^{\frac{-t+s}{\theta}-d}\\
    &= \infty,
  \end{align*}
  whenever $s\geq t+d\theta$. Therefore, $\fs\mu\leq\fd\mu + d\theta$, as required.
\end{proof}

From \cite[Corollary~3.3]{Fra22} we know that the bound obtained in Proposition~\ref{propo:app1} is sharp for measures with Fourier dimension $0$. We believe this is also the case for measures with positive Fourier dimension. One  way of building sharp examples could be to consider the convolution of two measures given by the following lemmas.

For the first lemma recall that the definition of Frostman dimension of a Borel measure $\mu$ on $\rd$ is
\begin{equation*}
    \frd\mu = \sup \big\{ s\geq0 : \mu\big( B(x,r) \big)\lesssim r^s\text{ for }x\in\rd,~r>0\big\}.
\end{equation*}
For any finite compactly supported Borel measure $\mu$ the inequality $\frd\mu\leq\sd\mu$ holds.

\begin{lma}\label{lemma:existenceMeasure1}
  There exists a finite compactly supported Borel measure $\mu$ on $\rd$ with $\fd\mu = 0$ and $\sd\mu = d$.
\end{lma}
\begin{proof}
  By \cite[Example~7]{EPS15}, there exists a compact set $X\subseteq\R$ with $\fd X = 0$ and $\mathcal{L}^1(X) >0$. Let $\nu$ be the restriction of $\mathcal{L}^1$ to $X$ and consider the product measure $\nu^d$ on $\rd$. Then $\fd \nu^d = 0$, and since $\nu^d\big( B(z,r) \big)\leq\mathcal{L}^d\big( B(z,r) \big)\lesssim r^d$, its Frostman dimension is $\frd\nu^d \geq  d$. Proposition~\ref{propo:app1} gives $\sd\nu^d \leq d$, so it must be that $\sd\nu^d =d$. Letting $\mu \coloneqq \nu^d$ gives the desired result.
\end{proof}

For the second lemma recall that a set $X\subseteq\rd$ is Salem if $\fd X = \hd X$. We say also that a Borel measure $\mu$ is Salem if $\fd\mu = \sd\mu$. The following lemma gives the existence of Salem measures on $\rd$ of any prescribed dimension in $[0,d)$.
\begin{lma}\label{lemma:existenceMeasure2}
  For  $t\in[0,d)$ there exists a finite compactly supported Borel measure $\mu$ on $\rd$ that is Salem with dimension $t$.
\end{lma}
\begin{proof}
  The case $t=0$ is trivial. For $t\in(0,d)$, by \cite[Theorem~1.2]{FH23} there exists a  finite compactly supported Borel  measure $\mu$ with $\fd\mu \geq t$ such that $\hd\spt\mu = t$. Since $t<d$ then $\sd\mu \leq t$ and $\mu$ is a Salem measure with dimension $t$.
\end{proof}

  Now, to build a measure with Fourier dimension $t\in(0,d)$ that satisfies \eqref{eq:propo:app1} with equality we may consider the convolution of the measure $\mu_{1}$ given by Lemma~\ref{lemma:existenceMeasure1} with the measure $\mu_{2}$ given by  Lemma~\ref{lemma:existenceMeasure2}.

  The concavity of $\fs\mu_{1}$ and Proposition~\ref{propo:app1} imply that for all $\theta\in[0,1]$, $\fs\mu_{1} = d\theta$. Since $\fd\mu_{1}=0$, given   $\gamma>0$ there exists some sequence $(z_{n})_{n\in\N}$ such that
  \begin{equation*}
      \lim_{n\to\infty}\big| \widehat{\mu_{1}}(z_{n}) \big|^2 |z_{n}|^\gamma = \infty.
  \end{equation*}
  In the construction of $\mu_{2}$ in \cite{FH23}, there is some freedom which we believe can be leveraged so that $\mu_{2}$ has bad Fourier decay for the same sequence as $\mu_{1}$, but we do not prove this explicitly here. More precisely,  assume that $\mu_{2}$ may be modified  appropriately so that for all $s>t$, $\limsup_{n\to\infty}\big| \widehat{\mu_{2}}(z_{n}) \big|^2 |z_{n}|^s = \infty$ where $(z_{n})_{n\in\N}$  is the sequence given to us by $\mu_1$.    Then, given $s>t$ and $s-t>\gamma>0$, applying the convolution formula
  \begin{equation*}
      \limsup_{n\to\infty}\big| \reallywidehat{\mu_{1}*\mu_{2}}(z_{n}) \big|^2 |z_{n}|^s \geq \liminf_{n\to\infty} \big| \widehat{\mu_{1}}(z_{n}) \big|^2|z_{n}|^{\gamma} \limsup_{n\to\infty}\big| \widehat{\mu_{2}}(z_{n}) \big|^2|z_{n}|^{s-\gamma} = \infty,
  \end{equation*}
 and  so $\fd\mu_{1}*\mu_{2} \leq t = \fd \mu_2$.  Moreover, \cite[Theorem~6.1]{Fra22} gives
  \begin{equation*}
    \fs \mu_{1}*\mu_{2} \geq\sup_{\lambda\in[0,1]}\Big( \fd^{\lambda\theta}\mu_{1} + \fd^{(1-\lambda)\theta}\mu_{2} \Big).
  \end{equation*}
  In particular, setting $\lambda = 1$,
  \begin{equation*}
      \fs\mu_{1}*\mu_{2} \geq \fs\mu_{1} + \fd\mu_{2}\geq d\theta + \fd\mu_{1}*\mu_{2}.
  \end{equation*}
  Defining $\mu \coloneqq \mu_{1}*\mu_{2}$, by Proposition~\ref{propo:app1} the reverse equality also holds, showing that \eqref{eq:propo:app1} is indeed sharp.

\subsection{Bounds involving $\ell^{p}$ spaces}

Given a finite  Borel measure $\mu$ on $\rd$, it follows from H\"older's inequality that if $\widehat{\mu}\in L^{\frac{2p}{\theta}}(\rd)$ for some $\theta\in[0,1]$ and $p\in(1,\infty)$, then $\fs\mu\geq d\theta/p$. Using the representation of Theorem~\ref{thm:coeffs} we now obtain the same bound in the case where $\widehat{\mu}\in\ell^{\frac{2p}{\theta}}(\zd\backslash\{ 0 \})$, provided $\mu$  is  supported in $[\delta,1-\delta]^d$ for some $0<\delta<1/2$.
\begin{cor}\label{cor:lp}
  Let $\theta\in(0,1]$, $p\in(1,\infty)$, and $\mu$ be a finite Borel measure on $\rd$  supported in $[\delta,1-\delta]^d$ for some $0<\delta<1/2$.  If   $\widehat{\mu}\in\ell^{\frac{2p}{\theta}}(\zd\backslash\{ 0 \})$, then $\fs\mu \geq d\theta/p$.
\end{cor}
\begin{proof}
  Let $s>0$, $\theta\in(0,1]$, and $p\in(1,\infty)$ be such that   $\widehat{\mu}\in\ell^{\frac{2p}{\theta}}(\zd\backslash\{ 0 \})$. For  $p'$, the conjugate exponent of $p$, Theorem~\ref{thm:coeffs} and H\"older's inequality give
  \begin{align*}
      \J_{s,\theta}(\mu)^{1/\theta} &\approx \sum_{ z\in\zd\backslash\{ 0 \}} \big| \widehat{\mu}(z) \big|^{\frac{2}{\theta}}|z|^{\frac{s}{\theta}-d}\\
      &\leq \bigg( \sum_{  z\in\zd\backslash\{ 0 \}}\big| \widehat{\mu}(z) \big|^{\frac{2p}{\theta}} \bigg)^{\frac{1}{p}}\bigg( \sum_{  z\in\zd\backslash\{ 0 \}}|z|^{\big( \frac{s}{\theta}-d \big)p'} \bigg)^{\frac{1}{p'}}\\
      &\approx \bigg( \sum_{  z\in\zd\backslash\{ 0 \}}|z|^{\big( \frac{s}{\theta}-d \big)p'} \bigg)^{\frac{1}{p'}}.
  \end{align*}
The last sum is finite provided $s<d\theta/p$, and thus $\fs\mu\geq d\theta/p$.
\end{proof}


\end{document}